\documentclass{siamltex}
\usepackage{amsmath,amsfonts,graphicx}
\title{Semidiscrete Finite Element Analysis of Time Fractional Parabolic Problems: A Unified Approach}
\author{Samir Karaa\footnotemark[2]
}
\newtheorem{remark}{Remark}[section]

\newcommand{\Ba}{\partial_t^{\alpha}}

\newcommand{\RL}{^{R}\partial_t^{\alpha}}
\newcommand{\I}{\mathcal{I}}
\newcommand{\cL}{\mathcal{L}}

\newcommand{\bv}{{\bf w}}
\newcommand{\bs}{{\boldsymbol\sigma}}
\newcommand{\bz}{{\boldsymbol\zeta}}
\newcommand{\bx}{{\boldsymbol\xi}}

\newcommand{\bV}{{\bf W}}
\newcommand{\bH}{{\bf H}}

\begin{document}
\date{}
\maketitle

\maketitle

\renewcommand{\thefootnote}{\fnsymbol{footnote}}
\footnotetext[2]{Department of Mathematics and Statistics, Sultan Qaboos University,
 P. O. Box 36, Al-Khod 123, Muscat, Oman (skaraa@squ.edu.om). This research is supported by the Research Council of Oman grant RC/SCI/DOMS/16/01. It was submitted for publication June 13, 2017.}
\renewcommand{\thefootnote}{\arabic{footnote}}

\begin{abstract}
In this paper, we consider the  numerical approximation of time-fractional parabolic problems involving Caputo derivatives in time of order $\alpha$, $0< \alpha<1$. We derive optimal error estimates for semidiscrete  Galerkin FE type approximations for problems with smooth and nonsmooth initial data.  Our analysis relies on energy arguments and exploits the properties of the inverse of the associated elliptic operator. We present the analysis in a general setting so that it is easily applicable to various spatial approximations such as conforming and nonconforming FEMs,  and FEM on  nonconvex domains. The finite element approximation in mixed form is also presented and  new error estimates are established for  smooth and nonsmooth initial data. Finally, an extension of our analysis to a multi-term time-fractional model is discussed.
\end{abstract}


\section{Introduction}
%\se

Let $\Omega$ be a bounded, convex polygonal domain in $\mathbb{R}^2$  with
boundary $\partial \Omega$ and let $T>0$ be a fixed value.  We are interested in the numerical approximation of the solution $u(x,t)$
of the following time-fractional  initial-boundary value problem: 
\begin{equation}\label{a}
\Ba u +\cL u=f\; \mbox{ in } \Omega\times (0,T],\quad u(0)=u_0\; \mbox{ in }  \Omega,\quad u=0 
\; \mbox{ on } \partial \Omega\times (0,T],
\end{equation}
where 
$
\cL u = -\mbox{div} [A(x)\nabla u]+c(x)u,
$
$f(x,t)$ is the forcing function and $u_0(x)$ is the initial data. 
Here, $A(x)=[{a_{ij}(x)}]_{1\leq i,j\leq 2}$ is a  $2\times 2$ symmetric and  uniformly positive definite 
in $\Omega$ matrix with smooth coefficients, and   $c(x)\in L^\infty(\Omega)$ is nonnegative in $\Omega$. 
In \eqref{a}, $\Ba$ denotes the
Caputo fractional derivative of order $\alpha$ ($0<\alpha<1$) with respect to $t$  defined by
\begin{equation} \label{Ba}
\Ba\varphi(t):=\I^{1-\alpha}\varphi'(t):=\int_0^t\omega_{1-\alpha}(t-s)\varphi'(s)\,ds\quad\text{with} \quad
\omega_{\alpha}(t):=\frac{t^{\alpha-1}}{\Gamma(\alpha)},
\end{equation}
where $\varphi'$ is the time derivative of $\varphi$ and $\I^{\nu}$ is the Riemann--Liouville time-fractional integral of order $\nu$. As $\alpha \rightarrow 1^-$,  $\Ba$  converges
to $u'$, and thus, problem \eqref{a} reduces to the standard parabolic problem. 
In analogy with Brownian motion of normal diffusion, the equation in \eqref{a} with $0<\alpha<1$ represents a macroscopic counterpart of time continuous random  walk \cite{GMMP-2002,MW-1965}.

In recent years, the model \eqref{a}  has received considerable attention, due to its great efficiency in capturing
%excellent capacity of  capturing 
 the dynamics of physical processes involving  anomalous transport phenomena.
 %involving  an anomalous transport mechanism. 
Several numerical schemes   have then been proposed with different types of spatial discretizations 
including finite difference, FE or spectral element methods, see \cite{CuestaLubichPalencia2006, ChenLiuAnhTurner2012,Cui2012,ZS2011, KMP2015, JLPZ2015, MustaphaMcLean2011,LML-2017}, and most recently, the finite volume element method \cite{KMP2015,KP-2017}. 
%    For the time discretization,
%  different time-stepping schemes (implicit and explicit) have been investigated including finite difference, %convolution quadrature,
%  and discontinuous Galerkin methods, see \cite{ChenLiuAnhTurner2012, CuestaLubichPalencia2006, Cui2012, Mustapha2011, %MustaphaMcLean2011, RZ2013, ZS2011}.

The main technical difficulty in designing robust numerical schemes and in carrying out a rigorous error analysis stems from the limited smoothing properties of the problem.  Specifically, for an initial data $u_0\in L^2(\Omega)$ and $f=0$, the following estimate  \cite{SY-2011}:
$$
\|\Ba u(t)\|_{ L^2(\Omega)}\leq Ct^{-\alpha}\|u_0\|_{ L^2(\Omega)},\qquad t>0,
$$
shows the singular behaviour of the solution $u$ near $t=0$.  Assuming high regularity on $u$ imposes additional  conditions on the given data $u_0$ and $f$, which are not, in general, reasonable. 
%Another difficulty is due to that fact that 
Note also that in the case of fractional-order evolution problems, 
the solution operators do not form a semigroup, as in the parabolic case,  so some useful techniques cannot be utilized. Attempts have then been made, using various techniques, including, spectral decomposition approach, Laplace transforms with the semigroup type theory, and novel energy arguments, to derive sharp error estimates for problem \eqref{a} under reasonable assumptions on the solution $u$.

%For instance,  a direct application of energy arguments to problem \eqref{a} does not lead to optimal convergence rates %even when the initial data $u_0 \in  H^2(\Omega)$.

Early papers dealing with optimal (with respect to data regularity) error estimates for time-fractional order problems consider the following subdiffusion equation
\begin{equation}\label{a1-n}
u_t(x,t)- {\RL}\Delta u(x,t)=f(x,t),
\end{equation}
which is closely related to but different from the model in \eqref{a}. Here, $\RL$ is the Riemann-Liouville fractional derivative in time defined by $\RL\varphi(t):=\frac{d}{dt}\I^{1-\alpha}\varphi(t)$.
In \cite{MT2010}, McLean and Thom\'ee 
established the first optimal $L^2(\Omega)$-error estimates for the  Galerkin FE solution of \eqref{a1-n}
with respect to the regularity of initial data using Laplace transform technique. Thus, they extended the  classical results obtained in \cite {BSTW-1977} for the standard parabolic problem.
% which has been thoroughly studied, see  \cite{thomee1997}. 
In \cite{McLeanThomee2010}, the  authors derived convergence rates  in the stronger $L^\infty(\Omega)$-norm. Recently, a delicate energy analysis has been developed in \cite{KMP2016}  to obtain similar estimates.

In recent papers \cite{JLZ2013,JLPZ2015,JLZ2016}, Jin {et al.} established optimal error estimates for the subdiffusion problem  \eqref{a}, with respect to the solution smoothness expressed through the problem data,
$f$ and $u_0$. In \cite{JLZ2013}, an approach based on Laplace transform and eigenfunction expansion  of the solution  has been exploited to derive {\it a priori} error estimates for the semidiscrete FEM applied to \eqref{a} with $f=0$.  
The  semidiscrete FEM for the inhomogeneous  equation with  a  weak right-hand side data $f$ has been considered in \cite{JLPZ2015}. In \cite{JLZ2016}, fully discrete schemes based on convolution quadrature in time  are derived and analyzed   for problems with smooth and nonsmooth data. 
%In \cite{JLZ2016}, a different approach in space has been proposed to improve some of the estimates establishes in %\cite{JLZ2013}.

%motivation

The first motivation of this work is to derive optimal error estimates for semidiscrete  Galerkin FE type approximations  to the problem  \eqref{a} on convex and nonconvex domains with both smooth and nonsmooth initial data, using energy arguments combined with a technique developed in \cite{BSTW-1977}, which is based on the inverse of the associated elliptic operator. We shall present our method in a general setting so that it can be extended to various discretizations in space, and can be easily adopted to different time-fractional problems. Thereby, we extend known results of the  parabolic case  to the fractional-order case $(0< \alpha<1)$. Our analysis depends on known properties of the associated elliptic problems, in contrast to the standard Laplace transform technique which relies on writing the corresponding semidiscrete problem in operator form. This procedure is not always feasible, e.g., in the analysis of the mixed form of problem \eqref{a}, and  can be complicated in other cases, such as, in the case of nonconforming FE approximations.
%specifically

% The method allows, in particular, the derivation of error estimates in various norms.
%and for different domains.
%Due to the  limited smoothing property of the solution $u$, a repeated use of the integral operator like %$\I\phi(t):=\int_{0}^t \phi(s)\;ds$ (see \cite{GP2011,GPY2014}) along with $t^m$ type weights is an essential tool to %provide optimal error  estimates. 

The second aim is to investigate a mixed form of  problem \eqref{a}, and derive optimal error estimates for 
the semidiscrete problem, using a standard Galerkin mixed FE method in space, for cases with  smooth and nonsmooth initial data. To the best of our knowledge, there is hardly any result for the mixed form of \eqref{a} except for  a recent paper \cite{2017}. In \cite{2017}, a non-standard mixed FE method is proposed and analyzed assuming higher order regularity  on the solution. Another related analysis for mixed method applied to the time-fractional Navier-Stokes equations is presented in \cite{XYZ-2017} where  high regularity assumptions on the exact solution are also made.

The rest of the paper is organized as follows. In Section \ref{sec:WRT}, we recall regularity properties of the solution $u$, and  state  some technical results.
In Section \ref{sec:AB}, we present our error analysis for the initial-boundary value problem \eqref{a}.  In Section \ref{sec:Appli},  applications are presented and optimal $L^2(\Omega)$-error estimates are established. 
%for cases with smooth and nonsmooth initial data. 
The applications include the standard $C^0$-conforming FE method defined on  convex and nonconvex domains, and some nonconforming methods. 
%we apply  our analysis to 
%a semidiscrete Galerkin  FE scheme for solving problem \eqref{a} on a convex polygonal domain, and 
%derive optimal $L^2(\Omega)$-error estimates for  smooth and nonsmooth initial data.
%Similar error estimates are also established  for nonconvex polygonal domains. 
In section \ref{sec:Mixed}, we introduce the mixed form of problem \eqref{a} and derive new error estimates for cases with smooth and nonsmooth data. Particularly relevant to this {\it a priori} error analysis is the appropriate use of several  properties of the time-fractional differential operator. Finally, in Section \ref{sec:multi},  we  discuss the extension of our analysis to a multi-term time-fractional model. 
%We show that the error estimates achieved for  problem \eqref{a} remain  valid for this model.

%%%%%%%%%%%%%%%%%%%%%%%%%%%%%%%%%%%%%%%%%%%%%%%%%%%%%%%%%%%%%%%%%%%%%%%%%%%%%%%%%%%%%%%
%%%%%%%%%%%%%%%%%%%%%%%%%%%%%%%%%%%%%%%%%%%%%%%%%%%%%%%%%%%%%%%%%%%%%%%%%%%%%%%%%%%%%%%
%%%%%%%%%%%%%%%%%%%%%%%%%%%%%%%%%%%%%%%%%%%%%%%%%%%%%%%%%%%%%%%%%%%%%%%%%%%%%%%%%%%%%%%

\section {Notation and  Preliminaries}\label{sec:WRT} 
%\section {Notation and technical results}\label{sec:WRT} 
%\se

We shall first introduce  notation and recall some preliminary results. Let $(\cdot,\cdot)$ be the inner product in $L^2(\Omega)$ and $\|\cdot\|$  the induced norm. 
The space $H^m(\Omega)$ denotes the standard Sobolev space with the usual norm $\|\cdot\|_{H^m(\Omega)}$.
%--------------------------------------------------------
%\subsection{The dotted spaces}
Let  $\{\lambda_j\}_{j=1}^\infty$ and $\{\phi_j\}_{j=1}^\infty$ denote respectively  the Dirichlet eigenvalues and 
eigenfunctions of  the symmetric and uniformly elliptic operator $\cL$ on the domain $\Omega$, with $\{\phi_j\}_{j=1}^\infty$ being an  orthonormal basis in $L^2(\Omega)$. For $r\geq 0$, we define the Hilbert space
$$
\dot H^r(\Omega)=\left\{v\in L^2(\Omega):\; \sum_{j=1}^\infty \lambda_j^r (v,\phi_j)^2<\infty\right\}
$$
equipped with the norm
$$
\|v\|_{\dot H^r(\Omega)} =\left(\sum_{j=1}^\infty \lambda_j^r (v,\phi_j)^2\right)^{1/2}.
$$
Note that $\|v\|_{\dot H^r(\Omega)}=(\cL^r v, v)^{1/2}=\|\cL^{r/2}v\|$. It is shown, see, for instance,  \cite[Lemma 3.1] {thomee1997}, 
that for $r$ a nonnegative integer, $\dot H^r(\Omega)$ consists of all functions $v$ in $ H^r(\Omega)$ which satisfy the boundary conditions $\cL^j v=0$ on $\partial \Omega$ for $j<r/2$, 
%and that the norms $\|\cdot\|_{\dot H^r(\Omega)}$ and  $\|\cdot\|_{H^r(\Omega)}$ are  equivalent in $\dot H^r(\Omega)$.
 and that the norm $\|\cdot\|_{\dot H^r(\Omega)}$ is equivalent to the usual norm in $H^r(\Omega)$.
For $r>0$ we also define  $\dot H^{-r}(\Omega)$ to be the dual space of $\dot H^r(\Omega)$.
% with respect to the inner product in $L^2(\Omega)$. 
 Since $\dot H^{1}(\Omega)$ and $H^1_0(\Omega)$ coincide, so does $\dot H^{-1}(\Omega)$ and $H^{-1}(\Omega)$, the dual space of $H^1_0(\Omega)$.
% defined by
%$$
%\dot H^{-r}(\Omega)=\left\{F\in H^{-1}(\Omega):\; \sum_{j=1}^\infty \lambda_j^r |\langle F,\phi_j\rangle|%^2<\infty\right\},
%$$
%and equipped with the norm
%$$
%\|F\|_{\dot H^{-r}(\Omega)}=\left(\sum_{j=1}^\infty \lambda_j^r |\langle F,\phi_j\rangle|^2\right)^{1/2}.
Note  that $\{\dot H^r(\Omega)\}$, $r\geq -1$,  form a Hilbert scale of interpolation spaces. 
%Hence, to link these dotted spaces to standard functional spaces,   
Thus, we denote  by $\|\cdot\|_{H^r_0(\Omega)}$ the norm on the interpolation scale between
$H^1_0(\Omega)$ and $L^2(\Omega)$ when $r\in[0,1]$ and  the norm on the interpolation scale between
and $L^2(\Omega)$ and $H^{-1}(\Omega)$ when $r\in[-1,0]$. 
Then, the norms $\dot H^{r}(\Omega)$ and $\dot H^{r}_0(\Omega)$  are equivalent for any $r\in[-1,1]$ by interpolation.

%--------------------------------------------------------
%\subsection{Regularity and technical results}
Regularity properties of the solution $u$ of the time-fractional problem \eqref{a} play a key role in the error analysis  of the finite element method,
particularly, since $u$  has singularity near $t=0$, even for smooth given data. 
From \cite{SY-2011} and \cite{Mclean2010}, we recall the regularity results for the problem \eqref{a} in terms of the initial data $u_0$  
for the homogeneous problem $(f=0)$. In particular,  for $t>0$,
%assuming zero source term $f$: for  $\ell =0,1$
%Over the convex domain $\Omega$, by combining the results of Theorems 4.1, 4.2   and 5.6 in \cite{Mclean2010}, for  $0\le r,\,\mu \le 2$, we obtain
\begin{equation}\label{eq:reg-a}
\|\Ba u\| + \|\cL u\|\le  Ct^{-\alpha}\|u_0\|,
%\|\Ba u\|_{L^2(\Omega)} + \|\cL u\|_{L^2(\Omega)}\le  Ct^{-\alpha}\|u_0\|_{L^2(\Omega)},
\end{equation}
and for $r\geq 0$, 
\begin{equation}\label{eq:reg}
t^\ell\|u^{(\ell)}(t)\|_{\dot H^{r+\mu}(\Omega)} \le  Ct^{-\alpha\mu/2}\|u_0\|_{\dot H^r(\Omega)}, \quad \ell=0,1,
\end{equation}
where  $0\leq \mu \leq 2$ when $\ell=0$ and $-2\leq \mu \leq 2$ when $\ell=1$.

Next, we recall some  properties of the fractional  operators  $\I^{\alpha}$ and $\Ba$ that will be used in the subsequent sections. For piecewise time continuous functions $\varphi:[0,T] \to X,$ where $X$ is a Hilbert space with inner product $(\cdot,\cdot)_X$ and norm $|\cdot|_X$, 
%$\varphi:[0,T] \to H_0^1(\Omega),$ we have
it is well-known that
$\int_0^T (\I^\alpha\varphi,\varphi)_X\,dt\geq 0$. Furthermore, by \cite[Lemma A.1]{McLean2012}, it  follows that  for $\varphi\in W^{1,1}(0,T;X)$,  $\int_0^T (\RL\varphi,\varphi)_X\,dt \geq 0$.

Using the result in \cite[Lemma 3.1 (iii)]{MustaphaSchoetzau2014} and the inequality $\cos(\alpha \pi/2)\geq 1-\alpha$, we obtain the following continuity property of $\I^{\alpha}$: for suitable functions $\varphi$ and $\psi$,
\begin{equation}\label{eq:p-1}
\int_0^t (\I^\alpha\varphi,\psi)_X\,ds\leq \epsilon \int_0^t (\I^\alpha\varphi,\varphi)_X\,ds +
\frac{1}{4\epsilon(1-\alpha)}\int_0^t (\I^\alpha\psi,\psi)_X\,ds,\quad \mbox{for } \epsilon>0.
\end{equation}
In our analysis, we shall also make use of the following inequality which holds by combining Lemmas 2.1 and 2.2 in 
\cite{LMM-2016}: if $\varphi(0)=0$, then
\begin{equation}\label{eq:p-2}
|\varphi(t)|_X^2\leq \frac{t^\alpha}{\alpha^2}\int_0^t (\I^{1-\alpha}\varphi',\varphi')_X\,ds,\quad \mbox{for } t>0.
\end{equation}
Finally, we recall the following identity which follows from the generalized Leibniz formula: 
%for fractional derivatives, we state and show some identities
%for our subsequent use. % These identities will be used in our error analysis.
\begin{equation}\label{Leibniz-1}
\RL(t\varphi)=t\RL\varphi +\alpha\I^{1-\alpha}\varphi.
\end{equation}
%Note that since $\Ba \varphi= \RL(\varphi-\varphi(0))$
Since $\Ba \varphi(t) = \RL(\varphi(t)-\varphi(0))$, we see that
\begin{equation}\label{Leibniz-2}
\Ba(t\varphi)=t\Ba \varphi +\alpha\I^{1-\alpha}\varphi+t\omega_{1-\alpha}(t) \varphi(0).
\end{equation}

For the rest of the paper, $C$ is a generic constant that may depend on $\alpha$ and $T$, but is independent of the spatial mesh size element $h$.

%%%%%%%%%%%%%%%%%%%%%%%%%%%%%%%%%%%%%%%%%%%%%%%%%%%%%%%%%%%%%%%%%%%%%%%%%%%%%%%%%%%%%%%
%%%%%%%%%%%%%%%%%%%%%%%%%%%%%%%%%%%%%%%%%%%%%%%%%%%%%%%%%%%%%%%%%%%%%%%%%%%%%%%%%%%%%%%
%%%%%%%%%%%%%%%%%%%%%%%%%%%%%%%%%%%%%%%%%%%%%%%%%%%%%%%%%%%%%%%%%%%%%%%%%%%%%%%%%%%%%%%

\section{General error estimates}\label{sec:AB}
%\se

%We now introduce the solution operator $T$ of the corresponding elliptic problem
%$$
%\cL  u=f\; \mbox{ in } \Omega,\qquad u=0 \mbox{ on } \; \partial \Omega,
%$$
%as $w=Tf$. Then, $T:L^2(\Omega)\to H^2(\Omega)\cap H_0^1(\Omega)$ is compact, selfadjoint and positive. In term of $T$,

Given the elliptic problem
$$
\cL  u=f\; \mbox{ in } \Omega,\qquad u=0 \mbox{ on } \; \partial \Omega,
$$
with $f\in L^2(\Omega)$, we now define the solution operator  $T:L^2(\Omega)\to H_0^1(\Omega)$ by 
$$a(Tf,v)=(f,v)\quad \forall v\in H_0^1(\Omega),$$ where $a(u,v)=(A\nabla u,\nabla v)+(cu,v)$. Note that
$T:L^2(\Omega)\to L^2(\Omega)$ is compact, selfadjoint and positive definite.
 In term of $T$, we may write the initial-boundary value problem \eqref{a} as
\begin{equation}\label{pr-1}
T\Ba u +u = Tf,\quad t>0,\quad u(0)=u_0.
\end{equation}
For the purpose of approximating the solution of this problem,  let  $V_h\subset L^2(\Omega)$ be a family of finite dimensional spaces  that depends on $h$, $0<h<1$. We assume that we are given a corresponding family 
of linear operators $T_h:L^2(\Omega)\to V_h$ which approximate $T$. Then consider the semidiscrete problem: find $u_h(t)\in V_h$ for $t\geq 0$ such that
 \begin{equation}\label{pr-2}
T_h\Ba u_h +u_h = T_hf,\quad t>0 ,\quad  u_h(0)=u_{0h}\in V_h,
\end{equation} 
where $u_{0h}$ is some approximation to  $u_0$. We shall make the assumptions that
 $T_h$ is selfadjoint,  positive semidefinite on $L^2(\Omega)$ and  positive definite on $V_h$.
%(i) $T_h$ is selfadjoint,  positive semidefinite on $L^2(\Omega)$ and  positive definite on $V_h$.
% Since $T_h^{-1}$ exists on  $V_h$,  \eqref{pr-2} may  be solved uniquely for $t\geq 0$.  
Let $e(t)=u_h(t)-u(t)$ denote the error at time $t$. Then, 
by subtracting \eqref{pr-2} from  \eqref{pr-1}, we find that  $e$ satisfies 
\begin{eqnarray}\label{IVP-nn}
 T_h\Ba e(t)+e(t)= (T_h-T)(f-\Ba u)(t),\quad t>0.
\end{eqnarray}
With $\rho(t) = (T_h-T)(f-\Ba u)(t)$, we, thus, obtain
\begin{equation} \label{eq:T_h}
T_h\Ba e(t) +e(t) = \rho(t),\quad  t>0,
\end{equation}
 with an initial data $e(0)\in L^2(\Omega)$. 
%Recall that $T_h$ is  compact, selfadjoint and  positive semidefinite on $L^2(\Omega)$. 
Then, by using the positive square root of $T_h$, we deduce from the positivity property of $\I^\alpha$ that for $t>0$, 
\begin{equation}\label{positivity-1}
\int_0^t (T_h\I^\alpha\varphi,\varphi)\,ds\geq 0.
\end{equation}
and, similarly, 
\begin{equation}\label{positivity-2}
\int_0^t (T_h\,\RL\varphi,\varphi)\,ds \geq 0.
\end{equation}
%We point out that, as $T_h$ is a finite-rank  operator  and so it is a compact operator. 
%Note that that,  $T_h$ is a compact operator since it has a  finite-dimensional range. 

Our task now is to derive estimates for $e$ in terms of $\rho$.  We begin by proving the following result.
%We shall present the problem in a general setting so that it can be adopted to different situations. 

%\subsection{Main estimates}
%Then, the following result  holds
%Then, the following result  holds, whose proof  can be found in \cite{MT2010, MustaphaMcLean2011}.
%
%
\begin{lemma}\label{lem:AB-1} Let  $e\in C([0,T],L^2(\Omega))$ such that $T_h\Ba e(t) +e(t) = \rho(t)$ for $t>0$. Then 
\begin{equation}\label{e-1}
\|e(t)\|\leq \|e(0)\|+4\left(\|\rho(0)\|+\int_0^t\|\rho_t(s)\|\,ds\right).
\end{equation}
In addition, if $T_he(0)=0$, then
\begin{equation}\label{e-2}
\beta \int_0^t\|T_h\Ba e(s)\|^2\,ds + (1-\beta)\int_0^t\|e(s)\|^2\,ds\leq \int_0^t\|\rho(s)\|^2\,ds, \quad \beta= 0,1.
\end{equation}
\end{lemma}
\begin{proof} Form the $L^2(\Omega)$-inner product between \eqref{eq:T_h} and $e_t$ to find that
$$
(T_h \I^{1-\alpha} e_t,e_t) +\frac{1}{2}\frac{d}{dt}\|e\|^2 = (\rho, e_t)= \frac{d}{dt}(\rho, e)-(\rho_t, e).
$$
Integrate with respect to time and  observe that $\int_0^t(T_h \I^{1-\alpha} e_t(s),e_t(s))\,ds\geq  0$ by \eqref{positivity-1}. Then, it follows that
\begin{eqnarray*}
\|e(t)\|^2 & \leq & \|e(0)\|^2+2\left(\|\rho(t)\|\|e(t)\|+\|\rho(0)\|\|e(0)\|+\int_0^t\|\rho_t\|\|e\|\,ds\right)\\
& \leq & \sup_{s\leq t}\|e(s)\|\left(\|e(0)\|+2\|\rho(t)\|+2\|\rho(0)\|+2\int_0^t\|\rho_t\|\,ds\right)\\
& \leq & \sup_{s\leq t}\|e(s)\|\left(\|e(0)\|+4\|\rho(0)\|+4\int_0^t\|\rho_t\|\,ds\right).
\end{eqnarray*}
Here, we have used $\rho(t)=\rho(0)+\int_0^t\rho_t(s)\,ds$.
Now, \eqref{e-1} follows by replacing   $\|e(t)\|^2$ with  $\|e(t)\|\sup_{s\leq t}\|e(s)\|$ on the left-hand side. For the second estimate \eqref{e-2}, we form the $L^2(\Omega)$-inner product between \eqref{eq:T_h} and $e$ and obtain
%and note that $T_h \Ba e= T_h \bar\partial_t^{\alpha}e$  
$$
(T_h \Ba e,e) +\|e\|^2 = (\rho, e).
$$
Then, we integrate with respect to time and 
note that  $\int_0^t(T_h \Ba e(s),e(s))\,ds\geq  0$ since $T_he(0)=0$ to derive \eqref{e-2} for $\beta=0$. The estimate with $\beta=1$ follows analogously after taking  the $L^2(\Omega)$-inner product between \eqref{eq:T_h} and $T_h\Ba e$ and proceeding similarly. This completes the rest of the proof.
\end{proof}

\begin{remark}\label{rem:AB-1-d}
%We now estimate $T_h\I^{1-\alpha}e$. For this purpose 
We integrate \eqref{eq:T_h} over $(0,t)$, keeping in mind that 
$T_he(0)=0$ and noting that $T_h\I^{1-\alpha}e = T_h\Ba \tilde e$ to find that 
\begin{equation}\label{ee-1}
T_h\Ba \tilde e +\tilde e=\tilde \rho,\qquad \tilde e(t)=\int_0^te(s)\,ds.
\end{equation}
An application of Lemma \ref{lem:AB-1}  yields 
$
\|\tilde e\|\leq 4 \int_0^t\| \rho(s)\|\,ds,
$
and hence,
$$
\|T_h\I^{1-\alpha}e\|\leq \|\tilde e\|+  \|\tilde\rho\|\leq C \int_0^t\| \rho(s)\|\,ds.
$$
This implies  $\|T_h\I^{1-\alpha}e\|^2\leq C t  \int_0^t\| \rho(s)\|^2\,ds$. Again, using Lemma Lemma \ref{lem:AB-1}, we deduce
\begin{equation}\label{e-2b}
\beta \int_0^t\|T_h\Ba \tilde e(s)\|^2\,ds + (1-\beta)\int_0^t\|\tilde e(s)\|^2\,ds\leq \int_0^t\|\tilde\rho(s)\|^2\,ds, \quad \beta\in [0,1].
\end{equation}
\end{remark}

%\newpage % 2

\begin{lemma}\label{lem:AB-2}
Let $e\in C([0,T],L^2(\Omega))$ such that 
\begin{equation} \label{IVP}
T_h\Ba e(t) +e(t) = \rho(t), \quad t>0, \quad T_he(0)=0.
\end{equation}
Then
\[
\int_0^ts^2\|e(s)\|^2\,ds
\le 2\int_0^t \left( s^2\|\rho\|^2+ 4 \|\tilde \rho\|^2\right)\,ds.
\]
\end{lemma}
\begin{proof} 
Multiply \eqref{eq:T_h} by $t$ and use the identity \eqref{Leibniz-2} so that
\begin{equation} \label{eq:T_h-t}
T_h\Ba (te) +te = t\rho+\alpha T_h\I^{1-\alpha}e.
\end{equation}
Form the $L^2(\Omega)$-inner product between \eqref{eq:T_h-t} and $te$, and integrate over $(0,t)$. Then, a use of
the positivity property \eqref{positivity-2} shows
\[
\int_0^ts^2\|e(s)\|^2\,ds
\le 2\int_0^t \left( s^2\|\rho(s)\|^2+ \|T_h\I^{1-\alpha}e(s)\|^2\right)\,ds.
\]
Now, the result follows by using \eqref{e-2b} with $\beta=1$. This completes the rest of the proof.
\end{proof}

%\newpage % 3

\begin{lemma}\label{lem:AB-3}
Under the assumption of Lemma $\ref{lem:AB-2}$, there holds for $t>0$,  
\[
\|e(t)\|^2 \leq C \left( \|\rho(t)\|^2 + \frac{1}{t}\int_0^t (\|\rho\|^2+ \|s\rho_t\|^2)\,ds\right).
\]
\end{lemma}
\begin{proof} 
Take the $L^2(\Omega)$-inner product between \eqref{eq:T_h-t} and $(te)_t$ to find that
\begin{equation*}
\begin{split}
(T_h\I^{1-\alpha}& (te)_t,(te)_t) +\frac{1}{2}\frac{d}{dt}\|te(t)\|^2 
= (t\rho,(te)_t)+\alpha (T_h\I^{1-\alpha}e,(te)_t)\\
=& \frac{d}{dt}(t\rho,te)-((t\rho)_t,te)
+\alpha \frac{d}{dt}(T_h\I^{1-\alpha}e,te)
-\alpha (T_h\Ba e,te).
\end{split}
\end{equation*}
Integrate over $(0,t)$ and use the positivity property \eqref{positivity-1} to obtain
$$
\frac{1}{2} t^2\|e(t)\|^2\leq \|te\|\left(
\|t\rho\|+ \alpha \| T_h\I^{1-\alpha}e\|\right) +
\int_0^ts( \|\rho+s\rho_t\|+\alpha  \| T_h\Ba e\|)\|e\|ds.
$$
Hence, we derive
$$
t^2\|e(t)\|^2\leq C\left( t^2\|\rho\|^2+ \| T_h\I^{1-\alpha}e\|^2 +
t\int_0^t(   \|\rho\|^2+\|s\rho_t\|^2 + \| T_h\Ba e\|^2+\|e\|^2)ds\right).
$$
Note that,  by Lemma \ref{lem:AB-1}, we arrive at $\int_0^t(\|T_h\Ba e\|^2+\| e\|^2)\,ds\leq 2\int_0^t\|\rho\|^2\,ds$ and also  $\|T_h\I^{1-\alpha}e\|^2\leq C t  \int_0^t\| \rho(s)\|^2\,ds$. This  completes  the proof.
\end{proof}

We shall now prove the main result of this section. 
\begin{lemma}\label{lem:AB-4}
Under the assumption of Lemma \ref{lem:AB-2}, there holds for $t>0$,  
\begin{equation}\label{est}
t^2\|e(t)\|^2 \leq C \left( t^2\|\rho(t)\|^2 + \|\tilde \rho\|^2+
\frac{1}{t}\int_0^t (\|s^2\rho_t\|^2+\|s\rho\|^2+ \|\tilde \rho\|^2)\,ds\right).
\end{equation}
\end{lemma}
\begin{proof}
Note that from \eqref{eq:T_h-t},  
$$
T_h\Ba (te) + te = \eta,
$$
where $\eta=t\rho+\alpha T_h\Ba \tilde e$. 
Then, by the estimate in Lemma \ref{lem:AB-3},
\begin{equation}\label{m-15-n}
\|te(t)\|^2 \leq C \left( \|\eta(t)\|^2 + \frac{1}{t}\int_0^t (\|\eta\|^2+ \|s\eta_t\|^2)\,ds\right).
\end{equation}
Since $T_h\Ba \tilde e=\tilde\rho-\tilde e$, it follows that
\begin{eqnarray}\label{m-16-n}
 \|\eta(t)\|^2  & \leq & C \left(t^2 \|\rho(t)\|^2  +\|\tilde \rho(t)\|^2+ \|\tilde e(t)\|^2\right)\nonumber\\
& \leq & C \left(t^2 \|\rho(t)\|^2 + \|\tilde \rho(t)\|^2
+\frac{1}{t}\int_0^t (\|s\rho\|^2+\|\tilde\rho\|^2)\,ds\right),
\end{eqnarray}
where the last term is obtained by applying Lemma \ref{lem:AB-3} to \eqref{ee-1}. For the time derivative in the integral  on the right-hand side of \eqref{m-15-n}, we note using Lemma \ref{lem:AB-2} that
\begin{eqnarray}\label{m-17-n}
 \int_0^t\|s\eta_t(s)\|^2\,ds  & \leq & C \left( \int_0^t(\|s^2\rho_t(s)\|^2+\|s\rho(s)\|^2 + 
 s^2 \| T_h\Ba e\|^2)\,ds \right)\nonumber\\
& \leq &  C \left( \int_0^t(\|s^2\rho_t(s)\|^2+\|s\rho(s)\|^2 + 
 s^2 \|e(s)\|^2)\,ds \right)\nonumber\\
 & \leq &  C \left( \int_0^t(\|s^2\rho_t(s)\|^2+\|s\rho(s)\|^2 + 
  \|\tilde\rho(s)\|^2)\,ds \right).
\end{eqnarray}
On substitution of \eqref{m-16-n} and \eqref{m-17-n} in \eqref{m-15-n}, we arrive at \eqref{est} and this completes the  lemma.
\end{proof}

As an immediate consequence, we obtain the following lemma.

\begin{lemma}\label{lem:AB-5}
Under the assumption of Lemma $\ref{lem:AB-2}$, there holds for $t>0$,
\begin{equation}\label{sup}
\|e(t)\|\leq C t^{-1}\sup_{s\leq t} ( \|\tilde\rho(s)\|+s\|\rho(s)\|+ s^2\|\rho_t(s)\|).
\end{equation}
\end{lemma}

\begin{remark}\label{rem:1}
Note that the estimate \eqref{sup} is still valid in the limiting case $\alpha=1$, i.e.,
for the  parabolic problem. This estimate is established in \cite[Formula (3.16)]{thomee1997}. 
%Also, we observe that in the above analysis, it is possible to replace the $L^2(\Omega)$-inner product $(\cdot , \cdot)$ 
%by any inner or semi-inner product  $\langle \cdot , \cdot \rangle$ for which $\langle T_h w, w \rangle$ is nonegative.
\end{remark}

\begin{remark}\label{rem:1-n} 
In the above analysis, it is possible to replace the $L^2(\Omega)$-inner product $(\cdot , \cdot)$ 
by any inner or semi-inner product  $\langle \cdot , \cdot \rangle$ for which $\langle T_h w, w \rangle$ is nonegative. %To give an example, 
As an example,  note that since $T_h$ is selfadjoint and positive semidefinite on $L^2(\Omega)$, the followings 
$$
(v,w)_{-r,h}=(T_h^rv,w),\qquad \|v\|_{-r,h}=(T_h^rv,v)^{1/2}
$$
define a semi-inner product and a semi-norm. Applying these, for instance, in the proof of Lemma \ref{lem:AB-1}, 
the estimate \eqref{e-1}  becomes  
\begin{equation}
\|e(t)\|_{-r,h}\leq \|e(0)\|_{-r,h}+4\left(\|\rho(0)\|_{-r,h}+\int_0^t\|\rho_t(s)\|_{-r,h}\,ds\right).
\end{equation}
%Following \cite{thomee-1980},  this basic estimate can be used to prove certain superconvergence results.
%Such an estimate can be ued to demonstrate superconvergence results, see  for details.
%This basic estimate has been used in \cite{thomee-1980} to prove certain superconvergence results.
%This basic estimate can be used  to prove certain superconvergence results, see \cite{thomee-1980}.
This basic estimate has been used in \cite{thomee-1980}  to prove certain superconvergence results.
\end{remark}

%%%%%%%%%%%%%%%%%%%%%%%%%%%%%%%%%%%%%%%%%%%%%%%%%%%%%%%%%%%%%%%%%%%%%%%%%%%%%
%%%%%%%%%%%%%%%%%%%%%%%%%%%%%%%%%%%%%%%%%%%%%%%%%%%%%%%%%%%%%%%%%%%%%%%%%%%%%
%%%%%%%%%%%%%%%%%%%%%%%%%%%%%%%%%%%%%%%%%%%%%%%%%%%%%%%%%%%%%%%%%%%%%%%%%%%%%

\section{Applications: Galerkin FE methods}\label{sec:Appli}
%\section{Applications: primal FE formulation}\label{sec:Appli} 
%\se

In this part,  we present some applications of our analysis to approximate the solution of \eqref{a} by Galerkin FE methods, and derive optimal $L^2(\Omega)$-error estimates  for problems with smooth and nonsmooth initial data. The Galerkin  methods include the standard $C^0$-conforming FE method on both  convex and nonconvex domains, and some nonconforming methods. Other Galerkin approximation methods, such as  Galerkin spectral methods, are in many ways similar to Galerkin FE methods, as the main difference is in the choice of the finite-dimensional approximating spaces.
%In the following, we denote by $\|\cdot\|$ the standard norm in $L^2(\Omega)$.

%For convenience, and without lost of generality, we assume that $A=-\Delta$.

%To define the scheme, let  $\mathcal{T}_h$ be a family of regular triangulations (made of simplexes $K$) of
%the  domain $\overline{\Omega}$ and let $h=\max_{K\in \mathcal{T}_h}(\mbox{diam}K),$ where $h_{K}$ denotes
%the diameter  of the element  $K.$  

%--------------------------------------------------------------------------

\subsection{$C^0$-conforming FE method}\label{sec:GFEM}
%\subsection{Standard Galerkin FEM}\label{sec:GFEM}

The weak formulation for problem  \eqref{a} is to seek $u:( 0,T]\to H^1_0(\Omega)$ such that
\begin{equation} \label{weak}
(\Ba u,v )+ a(u,v )=  (f,v )\quad
\forall v\in H^1_0(\Omega), \quad t>0, \quad u(0)=u_0,
\end{equation}
%where  $a(u,v)=(\nabla u, \nabla v)$ for all $u,v\in H^1_0(\Omega)$. 
where $a(\cdot,\cdot)$ is already defined.
The  approximate solution $u_h$ will be sought in the finite element space 
$$V_h=\{v_h\in C^0(\overline {\Omega})\;:\;v_h|_{K}\;\mbox{is linear for all}~
K\in \mathcal{T}_h\; \mbox{and} \; v_h|_{\partial \Omega}=0\},$$
where  $\mathcal{T}_h$ is a family of shape-regular partitions of
the  domain $\overline{\Omega}$ into triangles $K$, with $h=\max_{K\in \mathcal{T}_h}h_K,$ where $h_{K}$ denotes the diameter  of the element  $K.$  
The semidiscrete Galerkin FEM for problem   for \eqref{a} is then defined as: find  $u_h:(0,T]\to V_h$ such that
\begin{equation} \label{semi}
(\Ba u_h,v_h)+ a(u_h,v_h)=  (f,v_h)\quad
\forall v_h\in V_h, \quad t>0, \quad u_h(0)=u_{0h},
\end{equation}
where $u_{0h}\in V_h$ is a suitable approximation of $u_0$. 
%Upon introducing the discrete operator $\cL_h:V_h\to V_h$ defined by
%$$
%(\cL_h u,v)= a( u, v)\quad \forall u,v\in V_h,
%$$
%the semidiscrete scheme \eqref{semi} is rewritten as
%\begin{equation} \label{semi-11}
%\Ba u_h+\cL_h u_h= P_h f,\quad t>0, \quad u_h(0)=u_{0h},
%\end{equation}

To derive error estimates, we  introduce some more notation. Let  problem $T:L^2(\Omega)\rightarrow 
H^2(\Omega)\cap H_0^1(\Omega)$ 
be the solution operator of the  elliptic problem corresponding to \eqref{a}, i.e., for $f\in L^2(\Omega)$, we define 
$Tf$ by 
\begin{equation}\label{T}
a(Tf,v)=(f,v)\quad \forall v \in H_0^1(\Omega).
\end{equation}
Then,  $T$ is a bounded, selfadjoint and positive definite operator on $L^2(\Omega)$.
%, see \cite{thomee1997,BSTW-1977}. 
%By applying $T$ to \eqref{a} we obtain the equivalent problem 
Note that from \eqref{weak}, 
$$ 
 a(u,v )=  (f-\Ba u,v )\quad\forall v\in H^1_0(\Omega),
$$
and hence, we have an equivalent formulation as
\begin{equation}\label{pb-1}
T\Ba u+u=Tf, \quad t>0, \quad  u(0)=u_0.
\end{equation}
Similarly, we let  $T_h:L^2(\Omega)\rightarrow V_h$ be the solution operator of the corresponding discrete elliptic problem:
\begin{equation}\label{T_h}
a(T_hf, \chi)=(f,\chi)\quad \forall \chi \in V_h.
\end{equation}
Then, \eqref{semi} is equivalently rewritten as
\begin{equation}\label{pb-2}
T_h\Ba u_h+u_h=T_hf, \quad t>0, \quad  u_h(0)=u_{0h}.
\end{equation}
The operator $T_h$ is selfadjoint, positive semidefinite  on $L^2(\Omega)$ 
and positive definite  on $V_h$, see \cite{BSTW-1977}, and satisfies the following property:
\begin{equation}\label{T-T}
\|\nabla^\ell(T_h-T)f\|\leq Ch^{2-\ell}\|f\|\quad \forall f\in L^2(\Omega), \quad \ell=0,1.
\end{equation}
Furthermore, it is easily verified that
$$T_h=T_hP_h \qquad \mbox{and} \qquad  T_h=R_hT,$$
where  $P_h$ is the orthogonal projection of $L^2(\Omega)$ onto $V_h$  defined by $(P_h v-v, \chi)= 0$ $\forall \chi\in V_h,$ and $R_h$ is the  Ritz projection $R_h : H_0^1(\Omega) \to V_h $ defined by the following relation:
%\begin{equation}\label{bar Ritz}
$a(R_h v-v, \chi)= 0$ $\forall \chi\in V_h.$
%\end{equation}
For $t\in (0,T]$, we define the projection error $\rho(t)=R_h u(t)-u(t)$. Then, $\rho$ satisfies the following estimates
\cite{Ciarlet-2002}: for $\ell=0,1,$
\begin{equation}\label{rho-1}
\|\rho^{(\ell)}(t)\|_{\dot H^j(\Omega)}\leq C h^{m-j} \|u^{(\ell)}(t)\|_{\dot H^m(\Omega)},\qquad~~j=0,1,~~m=1,2.
\end{equation}

Now we prove the following theorem. 
%With no loss of generality, we choose  $u_{0h}=P_hu_0$.
%
%
\begin{theorem} \label{thm:L2}
 Let $u$ and $u_h$  be the solutions of $(\ref{a})$ and $(\ref{semi})$, respectively,  with $f=0$ and
 $u_h(0)=P_h u_0$.  Then, for $u_0 \in \dot H^\delta(\Omega)$,  $0\le \delta \le 2$,
$$
 \|(u_h-u)(t)\| \leq
 C h^2 t^{-\alpha(2-\delta)/2}\|u_0\|_{\dot H^\delta(\Omega)}, \quad t >0.
%Ch^2 t^{-\alpha\mu/2}~~{\rm \,.
$$
\end{theorem}
\begin{proof} 
Let $e(t)=u_h(t)-u(t)$ be the error of the FE approximation at time $t$. Then, from \eqref{IVP} and \eqref{a}, the error   $e$ satisfies 
\begin{equation}\label{IVP-2}
T_h\Ba e+e=\rho, \quad \rho=(T_h-T)\cL u. %=(R_h-I)u.
\end{equation}
%Since $T_he(0)=0$, Lemma \ref{lem:AB-5} applies directly to \eqref{IVP-2}.
Note that, with $u_{0h}=P_hu_0$,  
$T_he(0)=0$ since $(e(0),\chi)= 0 \,\forall \chi \in V_h$. Hence, we are now in position to apply Lemma \ref{lem:AB-5}. By \eqref{T-T}, we deduce  
\begin{equation}\label{sup-2}
\|e(t)\| \leq C h^2 t^{-1}\sup_{s\leq t} ( \|\cL \tilde u(s)\| +s\|\cL u(s)\| + s^2\|\cL u_t(s)\|).
\end{equation}
Using the regularity property in \eqref{eq:reg}, we obtain for $u_0 \in \dot H^{\delta}(\Omega)$ with $0\leq \delta\leq 2$,
$$
\|u(s)\|_{\dot H^{2}(\Omega)}\leq C s^{-\alpha(2-\delta)/2}\|u_0\|_{\dot H^{\delta}(\Omega)}.
$$
Hence, since $2-\delta<2$, we find
$$
\|\tilde u(s)\|_{\dot H^{2}(\Omega)}\leq \int_0^t\|u(\xi)\|_{\dot H^{2}(\Omega)}\,d\xi\leq  C  s^{-\alpha(2-\delta)/2+1}\|u_0\|_{\dot H^{\delta}(\Omega)},
$$
and, similarly,
$$
s\|u_t(s)\|_{\dot H^{2}(\Omega)}\leq C s^{-\alpha(2-\delta)/2}\|u_0\|_{\dot H^{\delta}(\Omega)}.
$$
Combining these estimates with \eqref{sup-2} completes the proof.
\end{proof}

\begin{remark}\label{rem:0} By splitting the error 
$$u_h-u=(u_h-R_h u)+(R_h u-u)=:\theta+\rho,$$
 noting that $\|\theta(t)\| \leq \|(u_h-u)(t)\|+\|\rho(t)\|$, and using the Ritz projection bound in \eqref{rho-1} (with $j=1$ and $m=2$), we conclude that the estimate in Theorem \ref{thm:L2} is valid for $\theta$. 
Under the quasi-uniformity condition on $V_h$,  the inverse inequality $\|\nabla\theta(t)\|\leq Ch^{-1}\|\theta(t)\|$, and the estimate $\|\rho(t)\|_{H^1(\Omega)} \le  Ch\|u(t)\|_{H^2(\Omega)}\le 
Ct^{-\alpha(2-\delta)/2}\|u_0\|_{\dot H^\delta(\Omega)}$, which follows from  \eqref{rho-1} (with $j=1$ and $m=2$)  and the regularity property \eqref{eq:reg}, we obtain the following optimal error estimate in the $H^1(\Omega)$-norm:
\begin{equation}\label{eq:0}
 \|\nabla(u_h-u)(t)\| \leq
 C h t^{-\alpha(2-\delta)/2} \|u_0\|_{\dot H^\delta(\Omega)}\quad{\rm for}~~t \in (0,T]~~{\rm with}~~0\le \delta \le 2\,.
\end{equation}
%By removing the quasi-uniformity mesh assumption, this error bound remains valid for $0\le \delta\le 1,$  see Theorem %\ref{thm: H1 bound}.
\end{remark}

%Now we give an error estimate in the gradient norm.
%
%
%\begin{theorem} \label{thm:H1}
% Let $u$ and $u_h$  be the solutions of $(\ref{a})$ and $(\ref{semi})$ with $f=0$,
%and let $u_h(0)=P_h u_0$. Then, for $u_0 \in \dot H^\delta(\Omega)$, $1\le \delta \le 2$,
%$$
% \|\nabla(u-u_h)(t)\| \leq
% C h t^{-\alpha(2-\delta)/2}\|u_0\|_\delta\quad{\rm for}~~t \in (0,T].
%%Ch^2 t^{-\alpha\mu/2}~~{\rm \,.
%$$
%\end{theorem}
%\begin{proof}
%To prove the estimate, we consider the equation \eqref{IVP-2} in  $X=H_0^1(\Omega)$ with  
%$u_0\in X$ and $u_{0h}=P_hu_0\in X$.  
%The space $X$ is now equipped with inner product $\langle \varphi,\chi\rangle=(\nabla \varphi, \nabla \chi)$ and induced %norm $\|\cdot\|_1$. We note that
%$T_h$ is selfadjoint and positive definite on $X$: for all $\varphi,\chi\in X$,
%$$
%\langle T_h\varphi,\chi\rangle = (\nabla T_h\varphi, \nabla \chi)= (\nabla T_h\varphi, \nabla R_h\chi)
%= (\varphi, R_h\chi)=  (\nabla\varphi, \nabla TR_h\chi)= \langle \varphi,T_h\chi\rangle.
%$$
%In particular,
%$$
%\langle T_h\varphi,\varphi\rangle = \|\varphi\|^2.
%$$
%By  Lemma \ref{lem:AB-5} we have
%\begin{equation*}\label{sup-3}
%\|e(t)\|_1\leq C  t^{-1}\sup_{s\leq t} ( \|\tilde \rho(s)\|_1+s\|\rho(s)\|_1+ s^2\|\rho_t(s)\|_1).
%\end{equation*}
%The desired estimate follows then by \eqref{rho-1} (with $j=1$, $m=2$) and \eqref{eq:reg}.
%\end{proof}

%
\begin{remark}\label{rem:1-nn} For smooth initial data $u_0\in {\dot H}^2(\Omega)$, the estimate in Theorem \ref{thm:L2} remains valid when $u_h(0)=R_hu_0$. Indeed, let $\bar u_h$ denote the solution of \eqref{semi} with $\bar u_h(0)=R_hu_0$. Then, 
$\xi:=u_h-\bar u_h$ satisfies
$$
T_h\Ba \xi+\xi= 0, \quad t>0,\quad \xi(0)= P_hu_0-R_hu_0.
$$
Since $\xi\in V_h$, it follows that $\Ba \xi+\cL_h \xi= 0$, 
where $\cL_h$ is the discrete operator $\cL_h:V_h\to V_h$ defined by
$$
(\cL_h u,v)= a(u, v)\quad \forall u,v\in V_h.
$$
Then, a regularity result similar to \eqref{eq:reg}, yields
$$
\|\xi(t)\|\leq \|(P_h-R_h)u_0\|\leq Ch^2\|u_0\|_{\dot H^2(\Omega)}.
$$
The  $L^2(\Omega)$-estimate follows then by the triangle inequality. 
%The estimate in the gradient norm follows by using the inverse inequality $\|\nabla \xi\|\leq Ch^{-1}\|\xi\|$.
\end{remark}

\begin{remark}\label{rem:1-b}
Instead of imposing  Dirichlet boundary conditions  in \eqref{a} we could
have considered, for instance, homogeneous Neumann type boundary conditions.
Assuming in such a case that $c(x)\geq c_0>0$ a.e. in $ \Omega$, the operator $\cL$ is again positive definite,
so the spaces $\dot H^r(\Omega)$ may be defined in an analogous way.
%the eigenvalues of the corresponding elliptic problem are again positive so that the spaces $\dot H^r(\Omega)$ may be
%defined in the analogous way. 
According to \cite{Mclean2010}, the smoothing property \eqref{eq:reg} still holds and we may
again introduce $T$ and $T_h$ and then consider both problems \eqref{pb-1} and \eqref{pb-2}. 
The analysis covers this case of boundary conditions.
\end{remark}

%--------------------------------------------------------------------------
%--------------------------------------------------------------------------

\subsection{FE method on nonconvex domain}\label{sec:GFEM-b}

Our next target is to study the FE approximation in the the case when the domain $\Omega$ is a nonconvex polygonal domain in $\mathbb{R}^2$, with (for simplicity) exactly one reentrant angle $\omega\in(\pi,2\pi)$, and set $\beta = \pi/\omega \in(\frac{1}{2},1)$. 
For the special case of an L-shaped domain, $\omega=3\pi/2$ and $\beta =2/3$. 
It is well-know that for such a domain, the regularity of the solution of the elliptic problem
$\cL  u=f$ in $\Omega$, $u=0$ on $\partial \Omega$ 
in limited as a result of the singularity near the reentrant corner. 
%A regularity shift-theorem for this problem, known in kellogg [--], may be expressed as 
%$$
%\|u\|_{\dot H^{1+s}}\leq C\|f\|_{\dot H^{-1+s}}=C\|\Delta\|_{\dot H^{-1+s}}, \quad \mbox{ for }\; 0\leq s< \beta.
%$$
Furthermore, the optimal  FE error in $L^2(\Omega)$-norm for this problem is reduced from $O(h^2)$ to $O(h^{2\beta})$.  
Indeed, we have the following error estimate:
%\begin{equation}\label{NC-1-b}
%\|R_hu-u\|+h^\beta \|\nabla(R_hu-u)\|\leq C_s h^{2\beta} \|u\|_{\dot{H}^{1+s}(\Omega)}, \quad (\beta <s \leq 1),
%\end{equation}
\begin{equation}\label{NC-1}
\|T_hf-f\|+h^\beta \|\nabla(T_hf-f)\|\leq C_s h^{2\beta} \|f\|_{\dot{H}^{-1+s}(\Omega)}, \quad (\beta <s \leq 1),
\end{equation}
where $C_s$ depends on $s$, see \cite{CLTW-2006}. We shall now demonstrate  that, for the homogeneous problem, an optimal $O(h^{2\beta})$   error estimate 
holds for the semidiscrete approximation for smooth and nonsmooth initial data.

\begin{theorem} \label{thm:NC} 
 Let $u$ and $u_h$  be the solutions of $(\ref{a})$ and $(\ref{semi})$, respectively, with $f=0$ and  $u_h(0)=P_h u_0$.  Assume that $\Omega$ is nonconvex  and has exactly one reentrant angle. Then, we have, for $\beta <s \leq 1$, with $C=C_s$,
$$
 \|(u_h-u)(t)\| \leq
 C h^{2\beta} t^{-\alpha(1+s-\delta)/2}\|u_0\|_{\dot H^\delta(\Omega)},\quad 0\leq \delta \leq 1+s.
$$
%and if the mesh is quasi-uniform, then
%$$
% \|(u-u_h)(t)\| \leq
% C h^{\beta} t^{-\alpha(1+s-\delta)/2}\|u_0\|_\delta\quad 1\leq \delta \leq 1+s.
%$$
\end{theorem}
\begin{proof}
It is worth noticing that, while proving the estimate in Lemma \ref{lem:AB-5}, the convexity of $\Omega$
%smoothness of $\partial \Omega$ 
was not actually required. Using Lemma \ref{lem:AB-5} and   \eqref{NC-1}, we find that for $\beta <s \leq 1$
\begin{equation}\label{sup-2-nn}
\|e(t)\|\leq C_s h^{2\beta} t^{-1}\sup_{s\leq t} \left( \| \tilde u(s)\|_{\dot{H}^{1+s}(\Omega)}+s\| u(s)\|_{\dot{H}^{1+s}(\Omega)}+ s^2\|u_t(s)\|_{\dot{H}^{1+s}(\Omega)}\right).
\end{equation}
Recalling that 
$$\|u(t)\|_{\dot{H}^{s_1}(\Omega)}\leq C t^{-\alpha(s_1-s_2)/2}\|u(t)\|_{\dot{H}^{s_2}(\Omega)}, \quad 0\leq s_2\leq s_1,$$
%and using \eqref{NC-1}, we find 
%$$
%\|\rho(t)\|\leq  C h^{2\beta} \|u(t)\|_{\dot{H}^{1+s}}\leq C h^{2\beta}  t^{-\alpha(1+s-\delta)/2}\|u_0\|_\delta, 
%\quad 0\leq \delta \leq 1+s.
%$$
we then complete  the proof by following the arguments in the proofs of Theorems \ref{thm:L2}.
% and \ref{thm:H1}.
\end{proof}

We note that when $u_0$ is sufficiently regular, $u_0\in \dot H^{1+s}(\Omega)$, the error bounds are uniform in time, whereas,  for the nonsmooth data, $u_0\in L^2(\Omega)$, the  $L^2(\Omega)$-norm of the error is bounded as 
$ \|(u_h-u)(t)\| \leq C h^{2\beta} t^{-\alpha(1+s)/2}\|u_0\|_{L^2(\Omega)}$.

\subsection{Nonconforming FE methods}\label{sec:NC}
Now, we come to the error analysis of nonconforming FE methods for problem \eqref{a}. 
As an example, we consider the method by Crouzeix and Raviart \cite{CR-1973}, based on the nonconforming FE space
%As an example, we set $\cL=-\Delta$, and consider the method by Crouzeix and Raviart \cite{CR-1973}, based on the %nonconforming FE space
\begin{equation*}
\begin{split}
\widetilde V_h=\{ & v\in L^2(\Omega): u|_K \mbox{ is linear for all } K\in \mathcal{T}_h,\, v  \mbox{ is }\\
& \mbox{ continuous at the midpoints of the interior edges, }\\
& \mbox{ and } v=0 \mbox{ at the midpoints of edges on } \partial\Omega \}.
\end{split}
\end{equation*}
The discrete problem becomes: find $u_h(t)\in \widetilde V_h$ such that
\begin{equation}\label{pr-9}
(\Ba u_h,v)+a_h(u_h,v)=(f,v)\quad \forall v\in \widetilde V_h,
\end{equation}
where the bilinear form $a_h:\widetilde V_h\times \widetilde V_h \to\mathbb{R}$ is defined by 
\begin{equation}\label{a-CR}
a_h(u,v)=\sum_{K\in \mathcal{T}_h}\int_K(A\nabla u\cdot \nabla v+cuv)\,dx,
\end{equation}
%\begin{equation}\label{a-CR}
%a_h(u,v)=\sum_{K\in \mathcal{T}_h}\int_K\nabla u\cdot \nabla v\,dx =
%\sum_{K\in \mathcal{T}_h}(\nabla u,\nabla v)_K,\quad u,v\in \widetilde V_h,
%\end{equation}
with  associated broken norm 
%$$\|v\|_h=\sqrt{\sum_{K\in \mathcal{T}_h}(A\nabla u\cdot \nabla v)_K}.$$ Note
$$\|v\|_h=\left(\sum_{K\in \mathcal{T}_h}\int_K(A\nabla v\cdot \nabla v)\,dx\right)^{1/2}.$$
Note that $\|\cdot\|_h$ is indeed a norm on $\widetilde V_h$. Let  $T_h:L^2(\Omega)\rightarrow \widetilde V_h$ be the solution operator of the corresponding discrete elliptic problem:
\begin{equation*}\label{T_h-CR}
a_h(T_hf,\chi)=(f,\chi)\quad \forall \chi \in \widetilde V_h.
\end{equation*}
%so that \eqref{semi} is equivalent to
%\begin{equation}\label{pb-2}
%T_h\Ba u_h+u_h=T_hf, \quad t>0, \quad  u_h(0)=u_{0h}.
%\end{equation}
Then, since $a_h(\cdot,\cdot)$ is symmetric, the operator $T_h$ is selfadjoint and  positive semidefinite  on $L^2(\Omega)$: for all $f,g \in L^2(\Omega)$
$$(f,T_hg)=a_h(T_hf,T_hg)=(T_hf,g) \quad \mbox{ and }\quad (f,T_hf)=\|T_hf\|_h^2\geq 0,$$
and clearly, $T_h$ is  positive definite  on $\widetilde V_h$. 
 Furthermore, the following well-known estimate holds:
 % \cite{CR-1973}:
\begin{equation}\label{CR-2}
\|T_hf-Tf\|+h\|T_hf-Tf\|_h\leq Ch^2\|f\|.
\end{equation}
%Since, the $L^2$-projection on $\widetilde V_h$ satisfies (ii), we conclude that the three conditions (i)-(iii) are
With $u_{0h}$  being the the $L^2$-projection of $u_0$ on $\widetilde V_h$ so that $T_h(u_{0h}-u_0)=0$, we deduce that the error estimate in Theorem \ref{thm:L2}  holds true for the  Crouzeix-Raviart nonconforming FE solution $u_h$. 
%{\bf need Theorem}

%Next, we choose $u_0\in H^1(\Omega)$, $u_{0h}=P_hu_0\in V_h$ and  consider the equation \eqref{IVP-2} in  the space %$X=V_h+H_0^1(\Omega)$ equipped with the inner product   $(\cdot,\cdot)_h= a_h(\cdot,\cdot)$ for all $u,\,v\in X$.\\
%{\bf problems:} $(u,T_hv)_h\neq (T_hu,v)_h$ for $u,v \in V_h+H_0^1(\Omega)$.
%$$(u,T_hv)_h= a_h(u,T_hv)=a_h(\tilde  R_h u,T_hv)=(\tilde R_h u, v) =(\tilde R_h u,P_h v)=a_h(T_h\tilde R_h u,P_h v)
%$$
%$$(T_hu,v)_h= a_h(T_hu,v)%=(\tilde R_h u,P_h v)=a_h(T_h\tilde R_h u,P_h v)
%$$
Our analysis can also be applied  to  other nonconforming methods, including  Nitsche's method \cite{Nitsche1971} and the Lagrange multiplier method of Babuska \cite{Babuska1973}. In Nitsche's method,  the bilinear form   
$a_h(\cdot,\cdot)$ in \eqref{a-CR}, with $\cL=-\Delta$,  is given by
$$
a_h(u,\chi)=a(u,\chi)-\left\langle \frac{\partial u}{\partial n} ,\chi\right\rangle-
\left\langle u, \frac{\partial \chi}{\partial n}\right\rangle+\beta h^{-1}(u,\chi),
$$
where $\langle\cdot,\cdot\rangle$ denotes the inner product in $L^2(\partial \Omega)$, ${\partial u}/{\partial n}$ the conormal derivative on $\partial\Omega$ and $\beta$ a positive constant. Here, the finite element  space is defined as
$$
 \widetilde V_h =\{\chi\in C^0(\overline\Omega): \chi|_K\in P_1(K)\},
$$
without any boundary conditions imposed on $\partial\Omega$. 
%
%\begin{remark} In the standard Laplace transform technique, one has to write \eqref{pr-9} in operator form, which is not %an easy task.
%\end{remark}

%We note that our method can cover
%Note that our analysis can also be applied to the Lagrange multiplier method of Babuska, see [-].

%\begin{remark} 	Upon minor modifications, our analysis extends  to problems  with more general  elliptic second-order %differential operator 
%$$A u =-\nabla\cdot(\kappa(x) \nabla u) + c(x) u,$$
%where the tensor $\kappa(x):{\mathbb R}^2\rightarrow {\mathbb R}^{2\times2}$ is smooth and has a positive minimum %eigenvalue 
%$\lambda_1(\kappa(x))\geq \kappa_0$ for some $\kappa_0>0$  and  
%$c(x)\in L^\infty(\Omega)$ with $c(x)\geq c_0 >0$ almost everywhere. {\bf More examples and comments will be given}
%%The corresponding  bilinear form $a(\cdot,\cdot):H_0^1(\Omega)\times H_0^1(\Omega)\rightarrow \mathbb{R}$ 
%vbecomes
%%\begin{equation*}\label{bilinear-n}
%%a (w,\chi) = (\kappa(x) \nabla w,\nabla \chi)+(c(x)w, \chi)\;\;\;\forall \chi\in H^1_0(\Omega).
%%\end{equation*}
%\end{remark}

% \subsection{The Stokes problem}

% \subsection{Space-time fractional equations}

% \subsection{More general elliptic operators}

% \subsection{Spectral methods}

% \subsection{Other boundary conditions (Neumann, Mixed)}

% \subsection{Quart. page 138: fourth order operator}

%%%%%%%%%%%%%%%%%%%%%%%%%%%%%%%%%%%%%%%%%%%%%%%%%%%%%%%%%%%%%%%%%%%%%%%%%%
%%%%%%%%%%%%%%%%%%%%%%%%%%%%%%%%%%%%%%%%%%%%%%%%%%%%%%%%%%%%%%%%%%%%%%%%%%
%%%%%%%%%%%%%%%%%%%%%%%%%%%%%%%%%%%%%%%%%%%%%%%%%%%%%%%%%%%%%%%%%%%%%%%%%%

\section {Applications: Mixed FE  methods} \label{sec:Mixed}
%\section {Applications: Mixed FE formulation} \label{sec:Mixed}
%\section { Mixed finite element formulation} \label{sec:Mixed}
%\se

In this section, we consider  the mixed form of the problem (\ref{a}) and establish {\it a priori} error estimates for smooth and nonsmooth initial data. To simplify the presentation, we choose $\cL=-\Delta$. By introducing the new variable $\bs=\nabla u$, the problem can be formulated as
$$
\Ba u- \nabla \cdot \bs =f, \qquad \bs=\nabla u, \qquad u=0 \; \mbox{ on } \partial\Omega,
$$
with $u(0)=u_0$. Let
$H(div;\Omega)= \{\bs\in (L^2(\Omega))^2:\nabla\cdot\bs\in L^2(\Omega) \}$
be a Hilbert space equipped with norm $\|\bs\|_{\bV} =(\|\bs\|^2+\|\nabla\cdot\bs\|
^2)^{\frac{1}{2}}$.
Then, with  $V=L^2(\Omega)$ and $\bV= H(div;\Omega)$,  the weak mixed formulation of (\ref{a}) 
is defined as follows:  find $(u,\bs):(0,T]\to  V\times \bV$ such that
%--------------
\begin{eqnarray}\label{w1-m}
(\Ba u, v)- (\nabla\cdot \bs, v) &=& (f,v) \;\;\;\forall v \in V,\\
\label{w2-m}
(\bs, \bv) +  (u,\nabla\cdot \bv) &=& 0\ \;\;\;\forall \bv \in \bV,
\end{eqnarray}
%Since the matrix $A$  is uniformly positive definite, there exist two
%positive constants $a_0$ and $a_1$ such that
%\begin{equation}\label{apd}%-------A is positive definite
% a_0 \|\bs\|\leq \|\bs\|_{A^{-1}}\le a_1\|\bs\|,~~~\mbox{where}~~
% \|\bs\|^2_{A{-1}}:=(A\bs,\bs).
%\end{equation}
with $u(0)=u_0$.  Note that the boundary condition $u=0$ on $\partial\Omega$ is implicitly contained in 
\eqref{w2-m}. By Green's formula, we formally obtain $\bs = \nabla u$ in $\Omega$ and $u=0$ on $\partial \Omega$.

Well-posedness of problem \eqref{a} is established in \cite{SY-2011} based on a spectral decomposition approach. 
%see Theorems 2.1 and 2.2 in \cite{SY-2011}. 
%We have the following regularity results: 
In particular, for $u_0\in L^2(\Omega)$ and $f=0$, it is shown that the problem \eqref{weak} has a unique weak solution 
$u\in C([0,T],L^2(\Omega))\cap C((0,T],\dot H^2(\Omega))$ with $\Ba u \in  C((0,T],L^2(\Omega))$, see \cite[Theorem 2.1]{SY-2011}. The regularity results for the inhomogeneous problem with a vanishing initial data are given in \cite[Theorem 2.2]{SY-2011}.  The well-posedness of \eqref{w1-m}-\eqref{w2-m} can then be established
using the equivalence of the weak formulation \eqref{weak} and the mixed formulation \eqref{w1-m}-\eqref{w2-m} based on the results in \cite{SY-2011}.

%The well-posedness of \eqref{w1-m}-\eqref{w2-m} can be established
%using the equivalence of the weak formulation \eqref{weak} and the mixed formulation \eqref{w1-m}-\eqref{w2-m} based on %the results in \cite{SY-2011}.

%----------------------------------------------------------
%----------------------------------------------------------
%----------------------------------------------------------

\subsection{Semidiscrete mixed FE problem}\label{sec:M1}

For the semidiscrete mixed formulation corresponding to \eqref{w1-m}-\eqref{w2-m}, let, as before, ${\mathcal T}_h$ be a shape-regular partition of the polygonal convex domain $\bar \Omega$ into triangles $K$ of diameter $h_K$. 
Further, let $V_h$ and $\bV_h$  be appropriate finite element subspaces of $V$ and $\bV$  satisfying the
Ladyzenskaya-Babuska-Brezzi (LBB) condition. 
 For example, let $V_h$ and $\bV_h$ be the Raviart-Thomas spaces \cite{RT-1977} of index $\ell\geq 0$ defined by
 $$
 V_h=\{ v\in L^2(\Omega):\;v|_{K}\in P_{\ell}(K) \;\forall K\in {\mathcal T}_h\}
 $$
 and 
 $$
 \bV_h=\{ {\bf v} \in \bH(div,\Omega):\;{\bf v}|_{K}\in RT_{\ell}(K) \;\forall K\in {\mathcal T}_h\},
 $$
 where $RT_{\ell}(K)=(P_{\ell}(K))^2+{\mathcal \mathbb x}P_{\ell}(K),$ ${\ell}\geq 0$. We note that  
 high order Raviart-Thomas elements do not lead to optimal error estimates due to the limited solution regularity.
 Hence, we  shall consider only the case $\ell=1$. 
For more examples of these spaces including 
Brezzi-Douglas-Marini spaces and Brezzi-Douglas-Fortin-Marini spaces, etc., see \cite{BF-91}.

The  corresponding semidiscrete mixed finite element approximation is to seek a pair $(u_h,\bs_h):(0,T]\to  V_h\times \bV_h$ such that
\begin{eqnarray}\label{w1a-m}
(\Ba u_h, v_h)- (\nabla\cdot \bs_h, v_h) &=& (f,v_h) \;\;\;\forall v_h \in V_h,\\
\label{w2a-m}
(\bs_h, \bv_h) +  (u_h,\nabla\cdot \bv_h) &=& 0 \;\;\;\forall \bv_h \in \bV_h,
\end{eqnarray}
% where $P_h :L^2(\Omega)\rightarrow V_h$ denotes the $L^2$-projection defined by $(P_h v-v, \chi)= 0$ for all $\chi\in %V_h.$
with $u_h(0)=u_{0h}$, where  $u_{0h}$ an appropriate approximation of $u_0$  in $V_h$.  With bases for $V_h$ and 
$\bV_h$, the matrix form of the discrete problem is
\begin{equation*}\label{eq: 11regularity property}
\begin{aligned}
A\Ba U-B\Sigma = F, &\\
B^TU+D\Sigma=0,  & \quad \mbox{ for } t>0,\quad U(0) \mbox{ given, }
\end{aligned}
\end{equation*}
where $U$ and $\Sigma$ are vectors corresponding to $u_h$ and $\bs_h$. It is easily seen that the matrices $A$ and $D$ 
are  positive  definite. Eliminating $\Sigma$, we have the system of fractional ODEs
$$
A\Ba U +BD^{-1}B^TU = F,\quad \mbox{ for } t>0,\quad U(0) \mbox{ given, }
$$
which by standard results in fractional ODE theory has a unique solution, see \cite[Chapter 3]{KST-2006} .
%Since $W_h$ and $\bV_h$ are finite dimensional, from \eqref{w2a-m} we  can eliminate $\bs_h$ in the discrete level  
%by writing it
%in terms of $u_h.$ Therefore, substituting in \eqref{w1a-m}, we obtain  a system of time-fractional ODEs.
%Existence and uniqueness can then be proved easily  and hence, we skip the proof.

For $(u,\bs)\in V\times \bV$, we define the intermediate mixed projection as the pair $(\tilde{u}_h,\tilde{\bs}_h)\in V_h\times \bV_h$ satisfying
\begin{eqnarray}\label{ppp-1}
(\nabla\cdot (\bs-\tilde{\bs}_h), v_h) &=& 0 \;\;\;\forall v_h \in V_h,\\
\label{ppp-2}
( (\bs-\tilde{\bs}_h), \bv_h) +  (u-\tilde{u}_h,\nabla\cdot \bv_h) &=& 0 \;\;\;\forall \bv_h \in \bV_h.
\end{eqnarray}
Then,  the following estimates hold, see  for instance \cite[Theorem 1.1]{JT-1981},
\begin{equation}\label{est-1}
\|u-\tilde{u}_h\|\leq C h^2\|u\|_{H^2(\Omega)},\qquad \|\bs-\tilde{\bs}_h\|\leq Ch^s\|u\|_{H^{s+1}(\Omega)}, \; s=1,2,
\end{equation}
and on quasi-uniform meshes, 
\begin{equation}\label{est-2}
\|u-\tilde{u}_h\|_{L^{\infty}(\Omega)}\leq C h^s|\ln h|\,\|u\|_{H^{s+1}(\Omega)}, \; s=1,2.
\end{equation}

In our error analysis, we shall use the following result, see \cite[Lemma 1.2]{JT-1981}.
\begin{lemma}\label{lemma} There exists a constant $C$ such that
for any pair $(\theta_h,{\bf z})\in V_h\times (L^2(\Omega))^2$ satisfying
$$
({\bf z},\bv_h)+(\theta_h,\nabla\cdot\bv_h)=0 \;\;\;\forall \bv_h \in \bV_h,
$$
we have
$$
\|\theta_h\|_{L^\infty(\Omega)}\leq C |\ln h|\,\|{\bf z}\|.
$$
\end{lemma}

%----------------------------------------------------------
%----------------------------------------------------------
%----------------------------------------------------------

%{\bf Give the plan ...}

We now start deriving error estimates for smooth initial data using energy arguments. Since the problem has a limited smoothing property, integration in time with a $t$ type weight is an essential tool to provide optimal error estimates. 
%The generalized Leibniz formula play a key role in the analysis.  
This idea has been used in \cite{KMP2016} and  \cite{Mustapha2017} to derive optimal error bounds for problems \eqref{a1-n} and \eqref{a}, respectively. A similar approach applied to mixed finite element methods for parabolic problems has also been exploited  in \cite{GP-2011}.

\subsection{Error estimates with smooth initial data}\label{sec:M2}

For the error analysis, define $e_u=u_h-u$ and $e_\bs =\bs_h-\bs$. Then, from (\ref{w1-m})-(\ref{w2-m}) and (\ref{w1a-m})-(\ref{w2a-m}), 
$e_u$ and $e_\bs$ satisfy the following equations
\begin{eqnarray}\label{ee1}
(\Ba e_u, v_h)- (\nabla\cdot e_\bs, v_h) &=& 0 \;\;\;\forall v_h \in V_h,\\
\label{ee2}
( e_\bs, \bv_h) +  (e_u,\nabla\cdot \bv_h) &=& 0 \;\;\;\forall \bv_h \in \bV_h.
\end{eqnarray}
To derive  {\it a priori} error estimates for the semidiscrete FE problem (\ref{w1a-m})-(\ref{w2a-m}),  we split the errors
$$e_u:=(u_h- \tilde{u}_h)-(\tilde{u}_h- u)=:\theta-\rho,$$
$$e_\bs:=(\bs_h- \tilde{\bs}_h)-( \tilde{\bs}_h-\bs)=:\bx-\bz.$$
From (\ref{ee1})-(\ref{ee2}), we note that $\theta$  and $\bx$ satisfy 
\begin{eqnarray}\label{aa}
(\Ba e_u, v_h)- (\nabla\cdot \bx, v_h) &=& 0 \;\;\;\forall v_h \in V_h,\\
\label{bb}
(\bx, \bv_h) +  (\theta,\nabla\cdot \bv_h ) &=& 0 \;\;\;\forall  \bv_h \in \bV_h.
\end{eqnarray}

In the next lemma, we derive preliminary bounds for $e_u$ and $\bx$. To do so, we let $u_h(0)=P_h u_0$, where
 $P_h$ denotes here the $L^2$-projection of $V$ onto $V_h$. 
% {\it Some arguments are borrowed from 
% \cite{KMP2016, Mustapha2017}. Here, we follow the approach in \cite{KMP2016, Mustapha2017}, and extend the results to %the mixed form of the problem \eqref{a}}.
%$P_h :L^2(\Omega)\rightarrow V_h$ denotes 
%$P_h$ is the $L^2$-projection defined by %$(P_h v-v, \chi)= 0$ for all $\chi\in V_h.$
%
%
\begin{lemma}\label{lem:1e} %LLLLLLLLLLLLLLLLLLLLLLLLLLLLLLLLLLLLLLLLLLLLLLLL
For $0<t\leq T$, we have
\[
\int_0^t(\I^{1-\alpha}e_u,e_u)\,ds + \|\I\bx(t)\|^2
\le C\int_0^t|(\I^{1-\alpha}\rho,\rho)|\,ds.
\]
\end{lemma}
\begin{proof}
Integrate (\ref{aa}) over the time interval $(0,t)$ and use the identity 
$\I^{2-\alpha}v'(t) = \I^{1-\alpha}v(t)-\omega_{2-\alpha}(t)v(0)$ to obtain
\begin{equation} \label{s1}
(\I^{1-\alpha} e_u,v_h)-(\nabla\cdot\I\bx,v_h)=
\omega_{2-\alpha}(t)(e_u(0),v_h)\quad \forall v_h \in V_h.
\end{equation}
Since $u_h(0)=P_h u_0$,  $(e_u(0),\chi)= 0$, and  Therefore
\begin{equation} \label{s2}
(\I^{1-\alpha} e_u,v_h)-(\nabla\cdot \I\bx,v_h)=0\quad \forall v_h \in V_h.
\end{equation}
%Similarly, integrate (\ref{ee2}) yields
%\begin{equation} \label{sup-14 ph}
%(\alpha \I\bx, \bv_h) +  (\nabla\cdot \bv_h, \I\theta) = 0,\quad \forall~\bv \in \bV_h.
%\end{equation}
Now choose $v_h=\theta$ in (\ref{s2}) and $\bv_h=\I\bx$ in (\ref{bb}), and add the resulting equations 
to obtain after integration
$$
\int_0^t (\I^{1-\alpha} e_u,e_u)\,ds+  \int_0^t(\bx, \I\bx)\,ds=-\int_0^t (\I^{1-\alpha} e_u,\rho)\,ds.
$$
By the continuity of the operator $\I^{1-\alpha}$ in \eqref{eq:p-1} with $\epsilon=1/2$, we see that 
$$
\left|\int_0^t (\I^{1-\alpha} e_u,\rho)\,ds\right|\leq \frac{1}{2}\int_0^t (\I^{1-\alpha} e_u,e_u)\,ds
+C(\alpha)\int_0^t (\I^{1-\alpha} \rho,\rho)\,ds.
$$
Noting that $(\bx, \I\bx)=\frac{1}{2}\frac{d}{dt}\|\I\bx\|^2$, we deduce
$$
\int_0^t (\I^{1-\alpha} e_u,e_u)\,ds+  \|\I\bx\|^2\leq C\int_0^t |(\I^{1-\alpha} \rho,\rho)|\,ds.
$$
This completes the proof. \end{proof}

In the next lemma, we derive an upper bound for  $\theta$ and $\bx$.  This bound leads to optimal convergence rates in 
the $L^2(\Omega)$-norm of $e_u$ and $e_\bs$, and a quasi-optimal convergence rate in $L^\infty(\Omega)$-norm for 
$e_u$. 
%The  initial data is assumed to be in $\dot H^\delta(\Omega)$ where $\delta \in  [1,2]$, see Theorem \ref{thm:mixed-sm}.

%
\begin{lemma}\label{lem:2e} %LLLLLLLLLLLLLLLLLLLLLLLLLLLLLLLLLLLLLLLLLLLLLLLL
For $0<t\leq T$, we have
\[
\|\theta(t)\|^2+t^\alpha\|\bx(t)\|^2 \leq
 Ct^{\alpha-2} \int_0^t \left[ \|\I^{1-\alpha}\rho\|\,\|\rho\|+\|\I^{1-\alpha}\rho_1'\|\,\|\rho_1'\|\right] \,ds.
\]
\end{lemma}
\begin{proof} 
Multiply both sides of (\ref{aa})  by $t$ and use \eqref{Leibniz-2} to find with $\theta_1=t\theta$ and  $\bx_1=t\bx$ that
\begin{equation} \label{s3}
(\I^{1-\alpha} \theta_1',v_h)-(\nabla\cdot \bx_1,v_h)= (\I^{1-\alpha} \rho_1'+\alpha\I^{1-\alpha}e_u,v_h)\quad \forall v_h \in V_h.
\end{equation}
Next  multiply both sides of  (\ref{bb}) by $t$ and   differentiate with respect to time to arrive at
\begin{equation} \label{s4}
( \bx_1', \bv_h) +  (\theta_1',\nabla\cdot \bv_h) = 0 \quad \forall\bv_h \in \bV_h.
\end{equation}
Choose $v_h=\theta_1'$ in (\ref{s3}) and $\bv_h=\bx_1$ in (\ref{s4}), then add the resulting equations 
to obtain after integration
\begin{equation} \label{s5}
\int_0^t (\I^{1-\alpha} \theta_1',\theta_1')\,ds+  \int_0^t( \bx_1', \bx_1)\,ds=
-\int_0^t (\I^{1-\alpha} \rho_1',\theta_1')\,ds+\alpha\int_0^t (\I^{1-\alpha} e_u,\theta_1')\,ds.
\end{equation}
Note that $(\bx_1', \bx_1)=\frac{1}{2}\frac{d}{dt}\|t^2\bx\|^2$. Using the continuity of the operator $\I^{1-\alpha}$ and the estimate in Lemma \ref{lem:1e}, we obtain after simplification 
\begin{equation} \label{s6}
\int_0^t (\I^{1-\alpha} \theta_1',\theta_1')\,ds+  t^2\| \bx(t)\|^2 \leq C
\int_0^t |(\I^{1-\alpha} \rho_1',\rho_1')|\,ds+C\int_0^t |(\I^{1-\alpha} \rho,\rho)|\,ds.
\end{equation}
Then, the desired estimate follows from \eqref{eq:p-2}. This concludes the proof.
\end{proof}

Using the previous lemmas, we now derive optimal  error estimates for the semidiscrete mixed finite element problem 
 with smooth initial data.  
%In the next theorem, we derive optimal convergence results   of the FE scheme \eqref{semi}
%in the $L^2$-norm  for both smooth and nonsmooth initial data $u_0$. 
%For $u_0 \in \dot H^{\delta}(\Omega)$ with $0\le \delta\le 2,$ we show that the error is bounded by $C h^2 %t^{-\alpha(2-\delta)/2}$ for each  $t\in (0,T]$. Recall that,  $\dot H^{\delta}(\Omega)=\{v\in H^{\delta}%(\Omega):~v=0~{\rm on}~\partial \Omega\}$ for $1/2<\delta\le 2$, while $\dot H^{\delta}(\Omega)=H^{\delta}%(\Omega)$ for $0\le \delta <1/2$.  Noting that, in the limiting case $\alpha \rightarrow 1^{-}$, %we recover the convergence rates for the parabolic equation $u'-{\cal{L}} u=f$.
\begin{theorem} \label{thm:mixed-sm}
 Let  $(u,\bs)$ and $(u_h,\bs_h)$ be the solutions of \eqref{w1-m}-\eqref{w2-m} and 
 \eqref{w1a-m}-\eqref{w2a-m}, respectively, with $f=0$ and  $u_{0h}=P_hu_0$. 
 Then, for $u_0 \in \dot H^\delta(\Omega)$ with $\delta \in  [1,2]$,  the following error estimates hold:
\begin{equation}\label{es1}
 \|(u_h-u)(t)\|\leq
 C h^2 t^{-\alpha(2-\delta)/2}\|u_0\|_{\dot H^\delta(\Omega)}, \quad t>0,
\end{equation}
\begin{equation}\label{es2}
\|(\bs_h-\bs)(t)\| \leq C h^2 t^{-\alpha(3-\delta)/2}\|u_0\|_{\dot H^\delta(\Omega)},\quad t>0,
\end{equation}
and with an additional quasi-uniformity condition on the mesh,
\begin{equation}\label{es3}
 \|(u_h-u)(t)\|_{L^\infty(\Omega)}\leq
 C h^2|\ln h|\, t^{-\alpha(3-\delta)/2}\|u_0\|_{\dot H^\delta(\Omega)},\quad t>0.
\end{equation}

%for $t \in (0,T]$ and $0\le \delta \le 2$.
\end{theorem}
\begin{proof} Using the first estimate in  \eqref{est-1} and \eqref{eq:reg}, we find after integration 
that for $t\in (0,T]$,
\begin{eqnarray*}
\int_0^t \left[ \|\I^{1-\alpha}\rho\|\,\|\rho\|+\|\I^{1-\alpha}\rho_1'\|\,\|\rho_1'\|\right] \,ds
&\leq & Ch^4 \int_0^t s^{1+\alpha(\delta-3)}\,ds \|u_0\|_{\dot H^\delta(\Omega)}\\
&  = & Ch^4 t^{2+\alpha(\delta-3)} \|u_0\|_{\dot H^\delta(\Omega)}, \quad \delta \in  [1,2].
\end{eqnarray*}
Then, from Lemma \ref{lem:2e}, it follows that
$$
\|\theta(t)\|\leq Ch^2 t^{-\alpha(2-\delta)/2} \|u_0\|_{\dot H^\delta(\Omega)}, \quad \delta \in  [1,2].
$$
The bound \eqref{es1} follows now from the decomposition $u_h-u=\theta-\rho$, and the estimate of $\rho$ in \eqref{est-1}. To establish \eqref{es2}, we first note that  Lemma \ref{lem:2e} and previous estimates yield 
$$
\|\bx(t)\|\leq Ch^2 t^{-\alpha(3-\delta)/2} \|u_0\|_{\dot H^\delta(\Omega)}, \quad \delta \in  [1,2].
$$
From   \eqref{est-1} and \eqref{eq:reg}, we arrive at
$$
\|\bz(t)\|\leq C h^2 \|u(t)\|_{H^3(\Omega)} \leq  C h^2t^{-\alpha(3-\delta)/2} \|u_0\|_{\dot H^\delta(\Omega)}, \quad \delta \in  [1,2].
$$
Then, \eqref{es2} follows from the decomposition $\bs_h-\bs=\bx-\bz.$ Finally, in order to show \eqref{es3}, 
we apply Lemma \ref{lemma} to \eqref{bb}  and  obtain,  by the quasi-uniformity of $\mathcal{T}_h$,
$$
\|\theta(t)\|_{L^\infty(\Omega)}\leq C|\ln h|\,\|\bx(t)\|.
$$
Hence, by Lemma \ref{lem:2e}, 
$$
\|\theta(t)\|_{L^\infty(\Omega)}\leq Ch^2  |\ln h| t^{-\alpha(3-\delta)/2} \|u_0\|_{\dot H^\delta(\Omega)}, \quad \delta \in  [1,2].
$$
Together with the estimate \eqref{est-2}, this completes the proof of \eqref{es3}.
%$$
%\|\rho(t)\|_{L^\infty(\Omega)}\leq Ch^2 \|u(t)\| \leq Ch^2  |\ln h| t^{-\alpha(3-\delta)/2} \|u_0\|_\delta, \quad %\delta \in  [1,2],
%$$
\end{proof}

%\newpage 
%----------------------------------------------------------
%----------------------------------------------------------
%----------------------------------------------------------
\subsection{Error estimates with nonsmooth initial data}\label{sec:M3}

Our next purpose is to derive error estimates for nonsmooth initial data, i.e., for $u_0\in L^2(\Omega)$. To this end, we combine our analysis developed in Section \ref{sec:AB} with the results of the previous subsection. 
For a given function $f\in L^2(\Omega)$, let $(u_h,\bs_h)\in V_h\times \bV_h$ be the unique solution of the mixed elliptic problem 
\begin{eqnarray}\label{w1-b}
- (\nabla\cdot \bs_h, v_h) &=& (f,v_h) \;\;\;\forall  v_h \in V_h,\\
\label{w2}
(\bs_h, \bv_h) +  (u_h,\nabla\cdot \bv_h) &=& 0 \;\;\; \forall \bv_h \in \bV_h.
\end{eqnarray}
Then, we define a pair of operators $(T_h,S_h):L^2(\Omega)\to V_h\times \bV_h$ as $T_hf=u_h$ and 
$S_hf=\bs_h$. With $T:L^2(\Omega) \to H^2(\Omega)\cap H_0^1(\Omega)$ being the solution operator of the continuous problem 
\eqref{a}, the following result holds  (see \cite[Lemma 1.5]{JT-1981}):
\begin{lemma}\label{lem:TT}
The operator $T_h:L^2(\Omega)\to V_h$  defined by $T_hf=u_h$ is selfadjoint, positive semidefinite on $L^2(\Omega)$ and
positive definite on $V_h$. Further
$$
\|T_hf-Tf\|\leq Ch^2\|f\|.
$$
\end{lemma}

We are now ready to prove the following nonsmooth data error estimates. In the proof, we need the following inverse
property:
% property:  assumption:
% (under the quasi-uniformity of the mesh): 
\begin{equation}\label{ie}
\|\nabla\cdot\bx\|\leq C h^{-1} \|\bx\|\quad \forall \bx\in \bV_h.
\end{equation}
%which holds under the quasi-uniformity of the mesh.
%
\begin{theorem}\label{thm:mixed-nsm} 
Let $(u,\bs)$ and $(u_h,\bs_h)$  be the solutions  of \eqref{w1-m}-\eqref{w2-m} and 
 \eqref{w1a-m}-\eqref{w2a-m}, respectively, with $f=0$ and  $u_{0h}=P_hu_0$. Then 
\begin{equation}\label{eq:mixed-1}
\|(u_h-u)(t)\|\leq Ch^2t^{-\alpha}\|u_0\|, \quad  t>0.
\end{equation}
%If the inverse inequality \eqref{ie} holds, then 
If the mesh is quasi-uniform,  then
\begin{equation}\label{eq:mixed-2}
\|(\bs_h-\bs)(t)\|\leq Cht^{-\alpha}\|u_0\|,\quad  t>0,
\end{equation}
and 
%and if the mesh is quasi-uniform,
\begin{equation}\label{eq:mixed-3}
\|(u_h-u)(t)\|_{L^\infty(\Omega)}\leq Ch|\ln h|t^{-\alpha}\|u_0\|,\quad  t>0.
\end{equation}
\end{theorem}
\begin{proof} From the definition of the operator $T_h$ above, the  semidiscrete problem may also be written as
$$T_h\Ba u_h+u_h=0, \quad t>0,\quad u_h(0)=P_hu_0.$$
Recalling the definition of the continuous operator $T$, we deduce that
%$$T\Ba u+u=0, \quad t>0,\quad u(0)=u_0,$$
$$
T_h\Ba e_u+e_u=-(T_h-T)\Delta u,\quad t>0,\quad T_he_u(0)=0.
$$
Since  $T_h$ satisfies the properties in Lemma \ref{lem:TT}, the estimate \eqref{eq:mixed-1}  follows immediately from Lemma \ref{lem:AB-5} and  the regularity result in \eqref{eq:reg}.
 In order to show  \eqref{eq:mixed-2}, we use \eqref{bb} and the inverse inequality \eqref{ie} to obtain
$$
 \|\bx\|^2\leq \|\theta\|\,\|\nabla\cdot\bx\|\leq Ch^{-1}\|\theta\|\,\|\bx\|.
$$
Since, by \eqref{eq:mixed-1} and \eqref{est-1}, $\|\theta\|\leq \|e_u\|+\|\rho\|\leq Ch^2t^{-\alpha}\|u_0\|$, it follows that
\begin{equation}\label{ac-1}
\|\bx(t)\|\leq Ch^{-1}t^{-\alpha}\|u_0\|.
\end{equation}
% In order to show  \eqref{eq:mixed-2}, we shall appeal again to Lemma \ref{lem:2e}.
%Using the first estimate \eqref{est-1} with $s=1$, we find  after integration 
%$$
%\int_0^t \left[ \|\I^{1-\alpha}\rho\|\,\|\rho\|+\|\I^{1-\alpha}\rho_1'\|\,\|\rho_1'\|\right] \,ds
%\leq  Ch^2 \|u_0\|^2 \int_0^t s^{1-2\alpha}\,ds 
%  \leq  Ch^2 t^{2-2\alpha}\|u_0\|^2.
%$$
Together with 
$$\|\bz(t)\|\leq Ch  \|u(t)\|_2\leq Ch  t^{-\alpha}\|u_0\|,$$ this establishes \eqref{eq:mixed-2}.
Finally, we derive \eqref{eq:mixed-3} by using the estimates  $\|\theta(t)\|_{L^\infty(\Omega)}\leq C|\ln h|\,\|\bx(t)\|$
and  \eqref{ac-1}. 
\end{proof}

\begin{remark}
The results in Theorems \ref{thm:mixed-sm} are optimal with respect to the polynomial degree and data regularity. 
In the limiting case  $\alpha=1$, we find the  bounds derived in \cite{JT-1981} for the parabolic problem.  The nonsmooth data error estimate \eqref{eq:mixed-1} established in Theorem \ref{thm:mixed-nsm} is also optimal, whereas the last two error bounds are not. 
This is due the  limited smoothing property of the time-fractional  equation. Note that due the presence of the
limited smoothing property, high order finite elements do not provide better error estimates in the case of nonsmmoth initial data. Finally, it is worth to mention that  the analysis of mixed methods extends to  problems on nonconvex domains.

%This is due to the presence of the  limited regularity of the exact solution.  
\end{remark}

%%%%%%%%%%%%%%%%%%%%%%%%%%%%%%%%%%%%%%%%%%%%%%%%%%%%%%%%%%%%%%%%%%%%%%%%%%%%%%%%%%%%%%%
%%%%%%%%%%%%%%%%%%%%%%%%%%%%%%%%%%%%%%%%%%%%%%%%%%%%%%%%%%%%%%%%%%%%%%%%%%%%%%%%%%%%%%%
%%%%%%%%%%%%%%%%%%%%%%%%%%%%%%%%%%%%%%%%%%%%%%%%%%%%%%%%%%%%%%%%%%%%%%%%%%%%%%%%%%%%%%%

\section{Multi-term time-fractional problem}\label{sec:multi}
%\section{ On extensions}
%\se

In this section, we briefly discuss the extension of our analysis to the following  multi-term time-fractional diffusion problem:
\begin{equation}\label{a-m}
P(\partial_t) u  + \cL u=f\; \mbox{ in } \Omega\times (0,T],\quad u(0)=u_0\; \mbox{ in }  \Omega,\quad u=0 
\; \mbox{ on } \partial \Omega\times (0,T],
\end{equation}
where  the multi-term differential operator $P(\partial_t)$ is defined by
\begin{equation*}\label{P}
P(\partial_t) = \partial_t^{\alpha}+\sum_{i=1}^m b_i\partial_t^{\alpha_i},
\end{equation*}
with $0<\alpha_m\leq\cdots\leq \alpha_1\leq \alpha<1$ being the orders of the fractional Caputo derivatives, and  $b_i>0$, $i=1,\ldots,m$. The multi-term differential operator $P(^R\partial_t)$ is defined analogously.
The model \eqref{a-m} was developed to improve the modeling accuracy of the single-term model 
\eqref{a}  for describing anomalous diffusion.
With the notation of Section \ref{sec:AB}, we consider the following initial value problem: 
\begin{equation} \label{eq:multi}
T_h P(\partial_t) e(t) +e(t) = \rho(t),\quad  t>0,
\end{equation}
with an initial data in $L^2(\Omega)$. The operator $T_h$ satisfies the conditions stated in Section \ref{sec:AB}.  
Then, we have the following result.
\begin{lemma}\label{lem:AB-7-n}
Let $e\in C([0,T],L^2(\Omega))$ satisfy \eqref{eq:multi} with $T_he(0)=0$. Then, there holds for $t>0$, 
\begin{equation*}\label{sup-multi}
\|e(t)\|\leq C t^{-1}\sup_{s\leq t} ( \|\tilde\rho(s)\|+s\|\rho(s)\|+ s^2\|\rho_t(s)\|).
\end{equation*}
\end{lemma}
\begin{proof} 
We first introduce the time-fractional integral operator $Q(\I)$  defined by
$$
Q(\I) = \alpha\I^{1-\alpha}+\sum_{i=1}^m \alpha_i b_i\I^{1-\alpha_i}.
$$
Then, results similar  to \eqref{positivity-1} and \eqref{positivity-2}, follow from the following positivity properties:
\begin{equation}\label{a-m-2}
\int_0^t (Q(\I)\varphi,\varphi)\,ds\geq 0
\quad \mbox{ and } \quad 
\int_0^t (P(^R\partial_t)\varphi,\varphi)\,ds \geq 0.
\end{equation}
%Furthermore, the continuity property in \eqref{eq:p-1} now reads:
%\begin{equation*}\label{eq:p-1}
%\int_0^t (Q(\I)\varphi,\psi)\,ds\leq \epsilon \int_0^t (Q(\I)\varphi,\varphi)\,ds +
%\frac{1}{4\epsilon \alpha_m}\int_0^t (Q(\I)\psi,\psi)\,ds,\quad \mbox{for } \epsilon>0.
%\end{equation*}
%where we have used the fact that $(1-\alpha)\leq (1-\alpha_i)$ for $i=1,\ldots,m$. 
%Finally, the generalized Leibniz formula takes the form, with $\varphi(0)=0$,
Furthermore, the generalized Leibniz formula takes the form: with $\varphi(0)=0$, 
\begin{equation}\label{Leibniz-2-m}
P(\partial_t)(t\varphi)=tP(\partial_t) \varphi +Q(\I)\varphi.
%+t\omega_{1-\alpha}(t) \varphi(0).
\end{equation}
Using \eqref{a-m-2} and \eqref{Leibniz-2-m}, we then prove Lemma \ref{lem:AB-7-n}  by following line-by-line the proofs of Lemmas \ref{lem:AB-1}-\ref{lem:AB-4} where $\Ba$ is replaced by $P(\partial_t)$ and  $\alpha \I^{1-\alpha}$ is replaced by $Q(\I)$.
\end{proof}

Regularity properties of the solution of problem \eqref{a-m} can be found in \cite{JLZ2015}. For $f=0$ and  $u_0\in \dot H^q(\Omega)$, $q\in [0,2]$, it is shown that (see \cite[Theorem 2.1]{JLZ2015})
$$
\|P(\partial_t)^\ell u(t)\|_{\dot H^p(\Omega)} \leq Ct^{-\alpha(\ell+(p-q)/2)} \|u_0\|_{\dot H^q(\Omega)},\quad t>0,
$$
where for $\ell=0$, $0\leq p-q\leq 2$ and for $\ell=1$, $-2\leq p-q\leq 0$. In addition to these results, one can 
verify that the solution of \eqref{a-m} satisfies the  regularity property stated in \eqref{eq:reg}.  
As an immediate consequence, we conclude that all the the error estimates achieved in Section \ref{sec:Appli} for problem \eqref{a} and in Section \ref{sec:Mixed} for the mixed form remain valid for the multi-term time-fractional problem \eqref{a-m} based on our analysis, with the only exception that some  minor modifications are needed in the proof of Theorem \ref{thm:mixed-sm}.
Theorem \ref{thm:L2} provides, in particular, an improvement of the nonsmooth data error estimate established in \cite[Theorem 3.2]{JLZ2015} where  an additional  log factor is involved. 

%\section{Concluding Remarks}
\section{Conclusions}
In this paper we provided a unified error analysis for semidiscrete time-fractional parabolic problems 
and derive optimal error estimates for both smooth and nonsmooth initial data. 
The analysis depends on known properties of the associated elliptic problems. Examples including spatial approximations by conforming and nonconforming Galerkin FEMs,  and by FEM on  nonconvex domains have been discussed.  Further examples, including space-time fractional parabolic equations  can  be considered.
Particularly interesting in this study, is the  mixed form which  fits within the framework of the present analysis. 
Error estimates in gradient and maximum norms  deserve further investigation. 
An interesting future research direction is the analysis of mixed finite element methods applied to the time-fractional Stokes equations.
%The analysis of mixed finite element methods applied to the time-fractional Stokes equations deserve also an %investigation.
% and will be part of our  future work.

% There are some issues that we did not address in this study, regarding 
%error estimates in maximum norm and negative norms, and superconvergence estimates. These will deserve further %investigations. 
%Mixed finite element methods for the time-fractional Stokes equations will be part of our  future work.

%\vspace*{0.6cm}
%There are further examples that we did not address in this study, such as, problems with Meumann or mixed boundary %conditions,  

\vspace*{0.4cm}

{\bf  Acknowledgements.} 
The author  thanks Prof. Amiya K. Pani for valuable comments and suggestions.
%The author would like to thank Prof. Amiya K. Pani for several useful comments and suggestions.

\end{document}